\documentclass[pdflatex,sn-mathphys-num]{sn-jnl}% Math and Physical Sciences Numbered Reference Style

\usepackage{amsmath,amssymb,amsfonts}
\usepackage{amsthm}
\usepackage{mathrsfs}
\usepackage{graphicx}
\usepackage{xcolor}
\usepackage{booktabs}
\usepackage{subcaption}   % for subfigures
\usepackage{caption}      % control caption format
\usepackage{enumitem}     % for (i)/(ii) lists
\setlist[enumerate]{itemsep=2pt,topsep=4pt}

\numberwithin{equation}{section}

\newtheorem{theorem}{Theorem}[section]
\newtheorem{lemma}[theorem]{Lemma}

\theoremstyle{definition}
\newtheorem{definition}[theorem]{Definition}

\theoremstyle{remark}
\newtheorem{remark}[theorem]{Remark}

\begin{document}

\title{Counting Zeros of Complex-Valued Harmonic Functions via Rouché's Theorem}

\author*{\fnm{Japheth} \sur{Carlson}}\email{jrc117@byu.edu}

\affil{\orgname{Brigham Young University Mathematics Department}, 
\city{Provo}, 
\state{UT}, 
\postcode{84602}, \country{USA}}

\abstract{Rouché's Theorem is among the most useful results in complex analysis for counting zeros of analytic functions. Rouché's Theorem also admits a harmonic analogue for counting zeros of complex harmonic functions. Previously, this analogue has been applied primarily to closed curves of simple geometry, such as circles, to count zeros. We demonstrate that non-circular critical curves can serve as effective contours by applying a harmonic Rouché-type argument to determine the total number of zeros of the complex harmonic family given by $f(z) = z^n + az^k + b\overline{z}^k - 1 $, where $n>k\geq1$ and $a,b > 0$. Under explicit inequalities relating $a$ and $b$, we determine the total number of zeros is either $n$ or $n+2k$ (counted with multiplicity). We also prove the zeros of $f$ are confined to the union of two explicit annuli in the plane: an inner annulus containing $k$ zeros and an outer annulus containing the remainder. }



\maketitle

\section{Introduction}\label{Introduction}

Consider the analytic polynomial
\[
\phi(z) = z^n + az^k - 1,
\]
where \(a\) is a positive real number and \(n\) and \(k\) are natural numbers with \(n > k\). By the Fundamental Theorem of Algebra, we know \(\phi\) has \(n\) zeros (counting multiplicities).

However, suppose we want to count the zeros of the complex \textit{harmonic} polynomial
\[
f(z) = z^n + az^k + b\overline{z}^k - 1,
\]
where $b$ is also a positive real number. Because \(f\) is not analytic and contains a variable and its conjugate, the number of zeros also depends upon the coefficient values of \(a\) and \(b\), as well as \textit{both} exponent values, $n$ and $k$. As a result, we are forced to look outside the Fundamental Theorem of Algebra to count the zeros of $f$.

To illustrate this point, contrast the functions

\[
f_1 = z^9 + z^4 + 0.5\overline{z}^4 - 1
\]

and
\[
f_2 = z^9 + 4.5z^4 + 7\overline{z}^4 - 1.
\]

The zeros of $f_1$ (Figure \ref{fig:plot1}) and $f_2$ (Figure \ref{fig:plot2}) are plotted and counted.

\begin{figure}[ht]
 \centering
 \begin{minipage}[b]{0.45\textwidth}
 \centering
 \includegraphics[width=\textwidth]{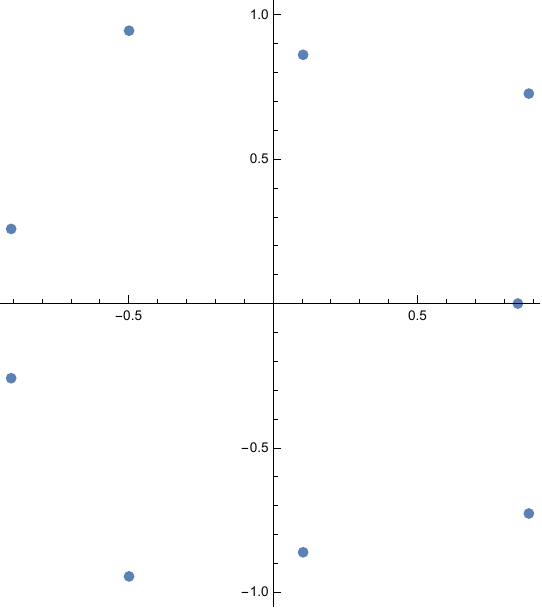}
 \caption{Zeros of $f_1$}
 \label{fig:plot1}
 \end{minipage}
 \hspace{0.05\textwidth}
 \begin{minipage}[b]{0.45\textwidth}
 \centering
 \includegraphics[width=\textwidth]{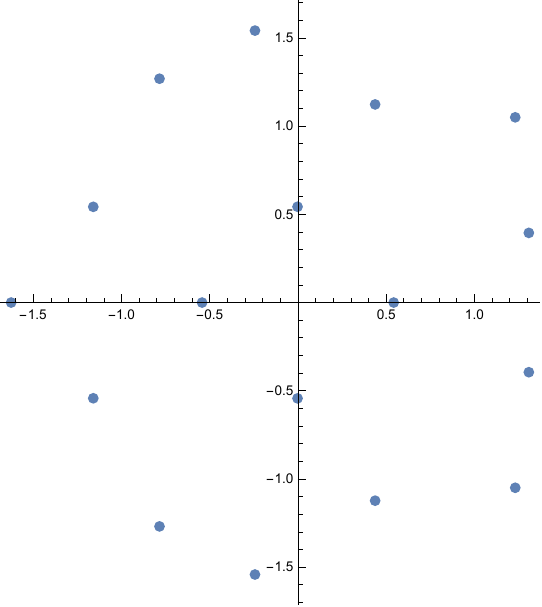}
 \caption{Zeros of $f_2$}
 \label{fig:plot2}
 \end{minipage}
\end{figure}

Even though $f_1$ and $f_2$ have the same exponents, $f_1$ has nine zeros while $f_2$ has 17 zeros. Clearly, the number of zeros is influenced not only by the degree, $n$, but by $a$ and $b$. This demonstrates the complicated and often unexpected behavior of complex harmonic polynomials.

With this example in mind, it is natural to wonder how to count the zeros of $f$ for different parameter values, as well as what the possible number of zeros of $f$ could be.

While investigating complex harmonic polynomials of degree $n$, Wilmshurst \cite{Wilmshurst} conjectured the valence, and therefore the count of the zeros, for such polynomials to be bounded above by $n^2$. At the same time, Wilmshurst constructed examples achieving this bound, indicating it is sharp. However, Wilmshurst's conjecture did not yet provide the tools to find the exact count of zeros or their location for specific harmonic polynomial families.

Among the first to count the zeros of a specific harmonic family were Brilleslyper et al. \cite{Brilleslyper}, who studied the harmonic trinomial family, denoted $p_c$, given by

\[
p_c(z) = z^n + c\overline{z}^k - 1.
\]

In their analysis, Brilleslyper et al. focused on a closed loop called the \textit{critical curve}, which is the set dividing the complex plane into two regions: (1) sense-\textit{preserving} and (2) sense-\textit{reversing}. This idea will be discussed further in Section \ref{Background}.

To count the zeros of $p_c$, Brilleslyper et al. successfully applied the argument principle to the critical curve of $p_c$. While this argument led to a precise zero-counting result for their family, it relied heavily upon the geometry of the critical curve of $p_c$ being circular and also provided little information as to the location of the zeros counted.

Subsequent researchers have studied harmonic trinomial families \cite{Brooks_Hudson, Brooks_Muthuprakash, Work}, including related families with poles \cite{Brooks_Lee, Lee}, but have similarly focused their efforts on families with circular critical curves or used other methods without directly analyzing the critical curve. Sandberg \cite{Sandberg} recently investigated a specific family of harmonic quadrinomials constructed so that the unit circle is always at least part of the critical curve, providing both the count of the zeros and bounds for their locations for given parameter values. Even so, results beyond harmonic trinomials and harmonic families with circular critical curves remained limited. For more information about the mathematical foundation of this field, see \cite{Melman, Wilmshurst}.

Also of interest, researchers have begun investigating and bounding the location of zeros of complex harmonic polynomials through various methods. Recently, Geleta and Alemu \cite{Geleta} derived the locations for the zeros of $p_c$ and provided a closed disk in which all zeros of complex-valued harmonic polynomials must exist. More detailed regions for harmonic trinomials were later found by Gao et al. \cite{Gao} and Melman \cite{Melman_location}, opening the door to similarly investigate complex harmonic quadrinomials with more complicated geometries.

Turning our attention back to our earlier example, we see the quadrinomial

\[
f(z) = z^n + az^k + b\overline{z}^k - 1
\]
is constructed by introducing a fourth term, $az^k$, to the family $p_c$ studied by Brilleslyper et al. Notably, the introduction of this term provides the simplest harmonic extension of $p_c$ that breaks circular symmetry of the critical curve, as evidenced in Figure \ref{CC_f1} and Figure \ref{CC_f2}.

More specifically, by separating $f$ into its analytic ($h$) and co-analytic ($\overline{g}$) parts, the critical curve of $f$ is defined by the set of all points in $\mathbb{C}$ where $|h'(z)| = |g'(z)|$, i.e., $|nz^{n-1} + akz^{k-1}| = kb|z|^{k-1}$, which forms weighted lemniscate-type curves similar to those previously analyzed by Melman \cite{Melman} in both analytic and harmonic settings. Similar geometries arise where multiple analytic and co-analytic monomials interact. Thus, introducing a second analytic monomial term in $f$ provides a natural minimal test case for determining whether the harmonic analogue of Rouché's Theorem may be used to extend zero-counting techniques beyond families of harmonic trinomials (or those with otherwise circular geometries) to more general harmonic polynomials.

\begin{figure}[htbp]
 \centering
 \begin{minipage}[b]{0.45\textwidth}
 \centering
 \includegraphics[width=\textwidth]{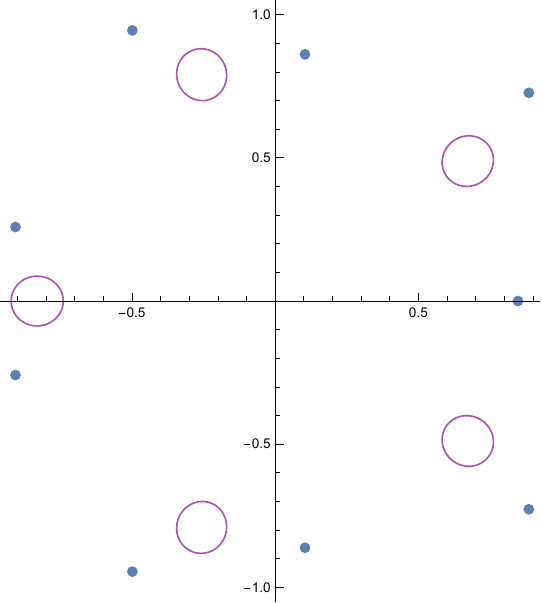}
 \caption{Critical Curve and Zeros of $f_1$}
 \label{CC_f1}
 \end{minipage}
 \hfill
 \begin{minipage}[b]{0.45\textwidth}
 \centering
 \includegraphics[width=\textwidth]{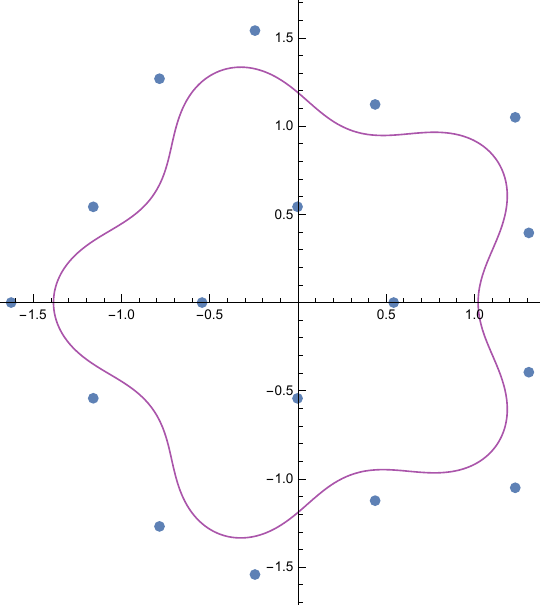}
 \caption{Critical Curve and Zeros of $f_2$}
 \label{CC_f2}
 \end{minipage}
\end{figure}

Because the behavior of harmonic polynomials with non-circular critical curves remains poorly understood, we address this gap by analyzing the complex-valued harmonic quadrinomial $f$. This paper contributes a generalization to count the zeros of a harmonic polynomial without circular symmetry and provides two annuli guaranteed to contain all zeros of $f$. We therefore focus our main theorems on (1) counting the zeros of $f$ (see Theorem \ref{TheoremA} and Theorem \ref{TheoremB}), and (2) identifying two annuli which must contain all zeros of $f$ (see Theorem \ref{Annulus}).\footnote{In this paper, we focus specifically on the case where $a \neq b$. The case where $a = b$ has recently been explored by Brooks et al. \cite{Brooks_Liechty}.} These theorems are given as follows:

\begin{theorem}
\label{TheoremA}
Let 
 \[
 f(z) = z^n + az^k + b\overline{z}^k - 1
 \]
 where $n$ and $k$ are natural numbers with $n>k$. Let $\epsilon > 0$. Suppose $0<a < (\frac{n-k}{n+k} - \epsilon)b$. Then, there exists $b_0$ such that, for all $b > b_0$, $f$ has $n+2k$ zeros.
 \end{theorem}
 
\begin{theorem}
\label{TheoremB}
Let 
 \[
 f(z) = z^n + az^k + b\overline{z}^k - 1
 \]
 where $n$ and $k$ are natural numbers with $n > k$. Suppose $0 < b < (\frac{n-k}{n+k} -\epsilon)a$. Then there exists an $a_0$ such that, for all $a > a_0$, $f$ has $n$ zeros.
\end{theorem}

\begin{theorem}
\label{Annulus}
 Let 
 \[
 f(z) = z^n + az^k + b\overline{z}^k - 1
 \]
where a, b $\in$ $\mathbb{R}^+$ satisfy $|b - a| > 2$ and $n$, $k$ $\in$ $\mathbb{N} $ with $n > k$ . Suppose without loss of generality $b > a$. Then there exist two annuli where the inner annulus, defined by radii $R_1$ and $R_2$, contains $k$ zeros and the outer annulus, defined by $R_3$ and $R_4$, contains all remaining zeros of $f$. Moreover, the radii are given by:

\begin{align*}
R_1 &= \frac{1}{(b + a + 1)^{\frac{1} k}}, \\
R_2 &= \frac{1}{(b-a-1)^{\frac{1}{k}}}, \\
R_3 &= (b - a - 1)^{\frac{1}{n-k}}\text{, and} \\
R_4 &= (b + a + 1)^{\frac{1}{n-k}}. 
\end{align*}
\end{theorem}

In Section \ref{Background}, we introduce the necessary definitions and results for our arguments. Three important results concern the critical curve, the order of zeros for a harmonic function, and Rouché's Theorem for Harmonic Functions. We proceed in Section \ref{Results} to prove each of our main theorems. Section \ref{Conclusion} provides a brief conclusion of our results and discusses possibilities for future research.

\section{Background}\label{Background}
Recall that any real-valued function $u(x, y)$ is \textit{harmonic} if it satisfies Laplace's equation:
\[
\Delta u = \frac{\partial^2 u}{\partial x^2} + \frac{\partial^2 u}{\partial y^2} = 0.
\]
A complex-valued function is harmonic if its real and imaginary parts are harmonic.

As Duren \cite{Duren} proved, any \textit{complex}-valued harmonic function $f$ in a simply connected domain $D \subset \mathbb{C}$ may be written as $f = h + \overline{g}$, where $h$ and $g$ are analytic in $D$. We call $h$ the \textit{analytic} part of $f$ and $\overline{g}$ the \textit{co-analytic} part of $f$.

By separating a complex-valued harmonic function $f$ into its analytic ($h$) and co-analytic ($\overline{g}$) parts, we define the \textit{sense-preserving} region of $f$ as the region in the plane where $|h'(z)| > |g'(z)|$. Similarly, the \textit{sense-reversing} region of $f$ is wherever $|h'(z)| < |g'(z)|$ \cite{Brilleslyper}.

While investigating these regions, Brilleslyper et al. \cite{Brilleslyper} discussed the concept of a \textit{critical curve} to describe the curve(s) in the plane separating the sense-preserving and sense-reversing regions. The critical curve is paramount in proving our zero-counting theorems (Theorem \ref{TheoremA} and Theorem \ref{TheoremB}) and is defined as follows:

\begin{definition}(Critical Curve)
The set of all $z \in \mathbb{C}$ where $|h'(z)| = |g'(z)|$ will be called the critical curve of $f$.
\end{definition}

Another important concept central to both the zero-counting and zero-locating theorems is the \textit{order of the zeros}. It can be shown that an equivalent definition to Duren's \cite{Duren} for the order of a zero states:

\begin{definition}(Order of the Zeros)
If $z_0$ is in the sense-preserving region, the order is the smallest $n \geq$ 1 such that $h^{(n)}(z_0) \neq 0.$ If $z_0$ is in the sense-reversing region, the order is $-m$, where $m \geq 1$ is the smallest such that $g^{(m)}(z_0) \neq 0$.
\end{definition}

Hence, a zero in the sense-preserving region of $f$ is assigned \textit{positive} order. Similarly, a zero in the sense-reversing region is assigned \textit{negative} order. Thereby, if the sum of the orders of the zeros in the plane for $f$ is $x$ and there are $y$ zeros of negative order, then there are $x + y$ zeros of positive order. Thus, the sum of the orders of the zeros is $x$, but there are $x + 2y$ zeros in total.

We first recall Rouché's Theorem for Analytic Functions \cite{Saff}.

\begin{theorem} (Rouché's Theorem for Analytic Functions)
Suppose that $p$ and $q$ are analytic in an open set containing a simple closed contour $C$ and its interior. If $|p(z)| > |q(z)|$ for all $z$ on \(C\), then $p$ and $p + q$ have the same number of zeros inside $C$.

\end{theorem}

% We now illustrate how Rouché's Theorem for Analytic Functions can be used to count the zeros of a complex analytic polynomial in Example \ref{Rouche_example}.

% \begin{example}
% \label{Rouche_example}

% The analytic polynomial

% \[
% \phi_1(z) = z^5 + 3z^2 + 1
% \]

% has five zeros inside the circle defined by $|z| = 2$.
% \end{example}

% \begin{proof}
% Let \(p(z) = z^5\) and \(q(z) = 3z^2 + 1\) so that \(\phi_1 = p + q\). Take \(C\) to be the closed contour \(|z| = 2\). Then, for $z$ on $C$,

% \[
% |p(z)| = |z^5| = |z|^5 = 2^5 = 32.
% \]

% Similarly, using the triangle inequality,
% \begin{align*}
%  |q(z)| &\leq |3z^2| + 1\\
% &= 3|z|^2 + 1\\
% &= 3(2)^2 + 1\\
% &= 13.\\
% \end{align*}

% Since \(p\) and \(q\) are analytic in the open set containing the simple closed contour \(C\) and its interior, and $|p(z)| > |q(z)|$ for all z on \(C\), $p$ and $p + q$ must have the same number of zeros inside $C$. Because $p$ has five zeros located within \(C\), $\phi_1$ also has five zeros in \(C\).

% \end{proof}

% \begin{figure}[ht]
%  \centering
%  \includegraphics[width=0.45\textwidth]{Figure_3.1.pdf}
%  \caption{Zeros of $\phi_1$ within $C$}
%  \label{zeros_phi}
% \end{figure}

% We confirm the proof of Example \ref{Rouche_example} graphically by plotting the five zeros of $\phi_1$ found within $C$ in Figure \ref{zeros_phi}.

% Similar to the previous example, we may count the zeros of $z^5 + \overline{z}^3 - 1$ by applying the harmonic analogue of Rouché's Theorem.

The harmonic analogue of Rouché's Theorem counts zeros with respect to their order and multiplicity, and is proved using a similar approach to the proof in the analytic setting by Brown and Churchill \cite{Brown}.

\begin{theorem}(Rouché's Theorem for Harmonic Functions)
Suppose that $f$ and $g$ are both harmonic in a simply connected domain $D \subset \mathbb{C}$. Let $C$ be a simple closed curve contained in $D$ not passing through a zero and let $\Omega$ be the open bounded region created by $C$. Suppose that $f$ and $g$ have no singular zeros in $D$ and let the sum of the orders of the zeros of $f$ and $g$ in $C$ be the sum of the orders of the zeros of $f$ and $g$ in $\Omega$ (counting multiplicities). If $|f(z)| > |g(z)|$ at each point on $C$, then the sum of the orders of the zeros of $f$ is equal to the sum of the orders of the zeros of $f + g$.
\end{theorem}
% This theorem is proved similarly to the proof of Rouché's Theorem by Brown and Churchill \cite{Brown}.

% We see Rouché's Theorem is a powerful tool in counting zeros within a closed curve. Notably, Rouché's Theorem has historically been used with primarily simple closed loops, such as the circle illustrated in Figure \ref{zeros_phi}.

We see Rouché's Theorem is a powerful tool for counting zeros within a closed curve. Notably, while Rouché's Theorem has typically been applied to simple closed loops such as circles, it may be used on any simple closed curve in $\mathbb{C}$ on which its hypotheses are satisfied.

Returning to $f$, it is not immediately clear how Rouché's Theorem can be used to count its zeros. However, if Rouché's Theorem can be applied to determine (1) the sum of the orders of the zeros of $f$, as well as (2) how many zeros lie within the sense-reversing region of $f$ (i.e., within its critical curve), then the total number of zeros of $f$ follows quickly. The first of these quantities is straightforward to obtain and is proved as Lemma \ref{order_of_zeros} in Section \ref{Results}.

The second quantity requires more attention. Consider once again Figure \ref{CC_f1} and Figure \ref{CC_f2}, given in Section \ref{Introduction}. We see the critical curve of $f$ may take different shapes and may have one or more closed loops. This shows that applying a traditional Rouché argument on circles of constant modulus is insufficient to count the zeros that lie within the critical curve.

Accordingly, the first pair of theorems in Section \ref{Results} is dedicated to counting the zeros of $f$ by applying Rouché's Theorem for Harmonic Functions directly to the critical curve of $f$. Our final theorem identifies two annuli that contain all zeros of $f$.

\section{Results}\label{Results}

In this section, we find the sum of the orders of the zeros in the plane, then the sum of the orders of the zeros in the critical curve, which counts the zeros in the sense-reversing region. We solve for the number of zeros in the sense-preserving region and therefore obtain the total number. We then deduce two annuli that must contain all zeros in the plane.

We begin by proving our zero-counting theorems, Theorem \ref{TheoremA} and Theorem \ref{TheoremB}. We then proceed by proving our zero-locating theorem, Theorem \ref{Annulus}.

\subsection{Zero Counting}\label{Zero_Counting}

\subsubsection{Proof of Theorem \ref{TheoremA}}\label{Proof_Theorem_A}

We now introduce several lemmas necessary to prove Theorem \ref{TheoremA}. Lemma \ref{order_of_zeros} finds the sum of the orders of the zeros for $f$ in the plane is $n$. Lemma \ref{k_zeros} applies a Rouché's Theorem argument to find when $k$ zeros exist within the critical curve of $f$. Because we already know the sum of the orders of the zeros of $f$, by counting the zeros within the critical curve (i.e., sense-reversing region) of $f$, we may quickly deduce how many zeros exist outside the critical curve (i.e., sense-preserving region) of $f$. Together, this gives us the total number of zeros of $f$. The final lemma, Lemma \ref{helpful_computation}, is used to resolve a tedious computation used in the proof of Theorem \ref{TheoremA}.

\begin{lemma}
\label{order_of_zeros}
Suppose $n > k$. Then, the sum of the orders of the zeros for \(f(z) = z^n + az^k + b\overline{z}^k - 1\) in the plane is \(n\).
\end{lemma}
 
\begin{proof}
Let $p(z) = z^n$ and $q(z) = az^k + b\overline{z}^k - 1$. Because \(z^n\) is the only term with the largest exponent, \(n\), we clearly find for $z$ sufficiently large that $|p(z)| > |q(z)|$. By applying Rouché's Theorem to a circle, $C$, of sufficiently large radius, we find there are \(n\) zeros of \(f\) within $C$. Since $C$ is any arbitrarily large contour in the plane, we conclude the sum of the orders of the zeros of $f$ is $n$.
\end{proof}

\begin{lemma}
\label{k_zeros}
Let \(f(z) = z^n + az^k + b\overline{z}^k - 1\), where $n$ and $k$ are natural numbers with $n > k$. Suppose \(a < b\). Define $h(z) = z^n + az^k - 1$ and $g(z) = bz^k$ such that \(f(z) = h(z) + \overline{g(z)}\). Then, whenever $|g(z)| > |h(z)|$ along the entire critical curve of $f$, there are \(k\) zeros in the sense-reversing region of \(f\).
\end{lemma}
 
\begin{proof}
Recall that the sense-reversing region of $f$ is defined by $|h'(z)| < |g'(z)|$. Differentiating gives $|h'(z)| = |nz^{n-1} + kaz^{k-1}|$ and $|g'(z)| = |kbz^{k-1}|$. Therefore, the sense-reversing region of $f$ is given by

\[
|nz^{n-1} + kaz^{k-1}| < |kbz^{k-1}|.
\]

By multiplying each side by \(|z^{-k+1}|\), we find
\[
|nz^{n-k} + ka| < bk.
\]

Examining the neighborhood around the origin of the complex plane, we find that because $nz^{n-k}$ is continuous and because $ak < bk$, there exists an $\epsilon$-neighborhood around the point $z = 0$ where

\[
|nz^{n-k} + ak| < bk
\]
holds. Simplifying, we find the sense-reversing region of $f$ is given by $ak < bk$.

Thus, a neighborhood around the origin is within the sense-reversing region of $f$. Hence, we find by applying Rouché's Theorem that there are $k$ zeros in the sense-reversing region of $f$.
\end{proof}

\begin{lemma}
\label{helpful_computation}
Suppose $0 < a < \lambda b$ for $0 < \lambda < \frac{99}{100}$ to be determined. Assume $n$ and $k$ are natural numbers with $n > k$. Then, for all \(\epsilon > 0\) there exists some $b_0$ such that, for all \(b > b_0\), \(\left(\frac{k(b - a)}{n}\right)^{-\frac{k}{n-k}} < \epsilon \).
\end{lemma}
 
\begin{proof}
Because we want to show that for any $\epsilon > 0$, there exists a $b_0$ such that for all \(b > b_0\), \(\left(\frac{k(b - a)}{n}\right)^{-\frac{k}{n-k}} < \epsilon \), it is sufficient to show that as $b$ approaches infinity, the term \(\left(\frac{k(b - a)}{n}\right)^{-\frac{k}{n-k}}\) goes to zero.

Now, because $\lambda < \frac{99}{100}$ and $a < \lambda b$, we have

\[
\frac{k(b - a)}{n} > \frac{k(b - \lambda b)}{n} = b\frac{k(1 - \lambda)}{n} \geq b\frac{k(1 - \frac{99}{100})}{n} = b\frac{k}{100n}.
\]

Since $n > k$, $\frac{-k}{n-k}$ is negative. So, $0\leq \left(\frac{k(b-a)}{n}\right)^\frac{-k}{n-k}$, which trends to 0 as $b$ trends to infinity.

\end{proof}

We are now prepared to prove Theorem \ref{TheoremA}.

\begin{proof}[Proof of Theorem \ref{TheoremA}.]
 Let \(h(z) = z^n + az^k - 1\) and \(g(z) = bz^k\) so that \(f(z) = h(z) + \overline{g(z)}\). Suppose \(0 < a < \lambda b\) for \(0 < \lambda < \frac{99}{100}\) to be determined.

 Observe that to be on the critical curve of $f$, the equality \( |h'(z)| = |g'(z)| \) must hold. That is,
 \[
 |nz^{n-1} + akz^{k-1}| = |bkz^{k-1}|.
 \]
 Multiplying each side by \(|z^{-k+1}|\) gives
 \[
 |nz^{n-k} + ak| = |bk| = bk.
 \]

 Thus, by the triangle inequality, for z on the critical curve,
 \[
 bk \leq n|z|^{n-k} + ak,
 \]
which is equivalent to
 \[
 k\frac{b-a}{n} \leq |z|^{n-k}.
 \]

 Similarly, by the reverse triangle inequality,
 \[
 bk \geq n|z|^{n-k} - ak,
 \]
 which is equivalent to
 \[
 k\frac{b+a}{n} \geq |z|^{n-k}.
 \]

 Thus, if \(z\) is on the critical curve,
 \begin{equation}\label{eq:crit-bounds}
 k\frac{b-a}{n} \leq |z|^{n-k} \leq k\frac{b+a}{n}.
 \end{equation}

 Because \(k\frac{b-a}{n} \leq |z|^{n-k}\), multiplying each side of the inequality by \(|z|^{k-n}\) finds
 \[
 |z|^{k-n}\left(k\frac{b-a}{n}\right) \leq 1.
 \]
 Because \(-\frac{k}{n-k} < 0\), raising each side to the \(-\frac{k}{n-k}\) power gives
 \[
 1 \leq |z|^k \left(k \frac{b-a}{n} \right)^{-\frac{k}{n-k}}.
 \]

 With bounds on \(a\), \( |z|^{n-k}\) and 1, we continue by finding where \(|h(z)|\) dominates \(|g(z)|\).

 On the critical curve, we find
 
 \begin{align*}
 |h(z)| &= |z^n + az^k -1|\\
 &\leq |z|^n + a|z|^k + 1\\
 &= |z|^k(|z|^{n-k} + a) + 1\\
 &\leq |z|^k\left(\frac{k(b+a)}{n} + a + \left(\frac{k(b-a)}{n}\right)^{-\frac{k}{n-k}}\right)\\
 &< |z|^k\left(\frac{k(b + \lambda b)}{n} + \lambda b + \left(\frac{k(b - a)}{n}\right)^{-\frac{k}{n-k}}\right)\\
 &= |z|^k\left(\frac{bk(1 + \lambda)}{n} + \lambda b + \left(\frac{k(b - a)}{n}\right)^{-\frac{k}{n-k}}\right).
 \end{align*}
 
 Fix $\epsilon > 0$. Then, by Lemma \ref{helpful_computation}, there exists some $b_0$ such that for all $b > b_0$, \((\frac{k(b - a)}{n})^{-\frac{k}{n-k}} < \epsilon\). Therefore, for $b > b_0$,
 \begin{align*}
 |h(z)| &< |z|^k\left(\frac{bk(1 + \lambda)}{n} + b\lambda + \epsilon\right).\\
 \end{align*}
 
 Thus, factoring \(b\) from the final term above, we find
 \[
 |h(z)| < b|z|^k\left(\frac{k(1 + \lambda)}{n} + \lambda + \frac{\epsilon}{b}\right).
 \]
 
 Thereby, $|h(z)| < |g(z)|$ if
 \[
 b|z|^k\left(\frac{k(1 + \lambda)}{n} + \lambda + \frac{\epsilon}{b}\right) < b|z|^k.
 \]
 
 Simplifying, we find
 \[
 k + k\lambda + \lambda n + \frac{n\epsilon}{b} < n.
 \]
 
 Solving for $\lambda$ gives
 \[
 \lambda < \frac{n-k}{n+k} - \frac{n\epsilon}{b(n+k)}.
 \]
 
Assuming \(b_0 > 1\), we restrict $\lambda$ such that
\[
\lambda < \frac{n-k}{n+k} - \epsilon.
\]
 
 Therefore, whenever \(a < (\frac{n-k}{n+k} - \epsilon)b \), there exists a \(b_0\) such that for all \(b > b_0\), the term \(|bz^k|\) dominates along the entire critical curve of \(f\). Then, by Lemma \ref{k_zeros}, there are $k$ zeros in the sense-reversing region of $f$.

 By Lemma \ref{order_of_zeros}, the sum of the orders of the zeros of \(f\) must equal \(n\). Because \(k\) zeros are located within the sense-reversing region of $f$, there must be \(n + k\) zeros in the sense-preserving region so that the sum of the orders of the zeros of \(f\) is \(n\). Therefore, for all \(b > b_0\), there are $n + 2k$ zeros of $f$.
\end{proof}

\subsubsection{Proof of Theorem \ref{TheoremB}}

We now count the zeros of $f$ where $a > b$. Because $a, b > 0$ and $|z|^k = |\overline{z}^k|$, all our estimates on the critical curve given in the proof of Theorem \ref{TheoremA} depend only on $a|z|^k$ and $b|z|^k$ (see \eqref{eq:crit-bounds}). Thus, our argument is invariant under swapping $a$ and $b$. The only change is that we must replace Lemma \ref{k_zeros} with Lemma \ref{no_zeros}, given below, when counting the zeros inside the critical curve.

\begin{lemma}
\label{no_zeros}
Let
\(f(z) = z^n + az^k + b\overline{z}^k - 1\).
where $n$ and $k$ are natural numbers with $n > k$. Suppose \(a > b\). Define $h(z) = z^n + b\overline{z}^k - 1$ and $g(z) = az^k$. Then, whenever $|g(z)| > |h(z)|$ along the entire critical curve of $f$, there exist no zeros within the critical curve of \(f\).
\end{lemma}
 
\begin{proof}
By Lemma \ref{k_zeros}, we have that the sense-reversing region of \(f\) in the $\epsilon$-neighborhood around the origin is given by

\[
ak < bk.
\]

Thus, by flipping the inequality, the sense-preserving region is given by
\[
ak > bk.
\]

Because $a > b$, we find that the neighborhood around the origin is within the sense-preserving region of \(f\). Hence, by applying Rouché's Theorem, we find there are no zeros within the sense-reversing region of $f$.
\end{proof}

We are now ready to prove Theorem \ref{TheoremB}, similar to Theorem \ref{TheoremA}.

\begin{proof}[Proof of Theorem \ref{TheoremB}.]

This proof follows from Theorem \ref{TheoremA} by reversing the roles of \(a\) and \(b\). Thus, we find whenever $b < (\frac{n-k}{n+k} - \epsilon)a$, there exists an $a_0$ such that, for all \(a > a_0\), the term \(|az^k|\) dominates along the entire critical curve of $f$.

By Lemma \ref{no_zeros}, none of the zeros appear within the sense-reversing region of $f$. Therefore, there are no zeros of $f$ within its critical curve; so, by Lemma \ref{order_of_zeros}, $f$ has exactly $n$ zeros.

\end{proof}

\subsection{Zero Locating}

\subsubsection{Proof of Theorem \ref{Annulus}}

We now work to find two explicit annuli in the complex plane which contain all zeros of $f$ whenever $|b - a| > 2$.

\begin{remark}
Note that in the case examined in Theorem \ref{TheoremA}, for $b$ sufficiently larger than $a$, we are guaranteed the inner annulus contains the $k$ zeros within the critical curve of $f$, while the outer annulus contains all $n + k$ zeros of $f$ outside its critical curve. Similarly, for $a$ sufficiently larger than $b$, all zeros of $f$ belong outside its critical curve, with the inner annulus containing $k$ zeros and the outer annulus containing the remaining $n - k$ zeros. Moreover, for any fixed $a$, we find that as $b$ increases, the ratios $\frac{R_1}{R_2}$ and $\frac{R_3}{R_4}$ converge to 1. Similarly, for any fixed $b$, as $a$ increases, the ratios $\frac{R_1}{R_2}$ and $\frac{R_3}{R_4}$ converge to 1. In other words, our inner and outer annuli converge to an inner and an outer circle upon which all zeros of $f$ are located. An illustration of these annuli is given by Figure \ref{annular_bounds}.

\begin{figure}[ht]
 \centering
 \includegraphics[width=0.45\textwidth]{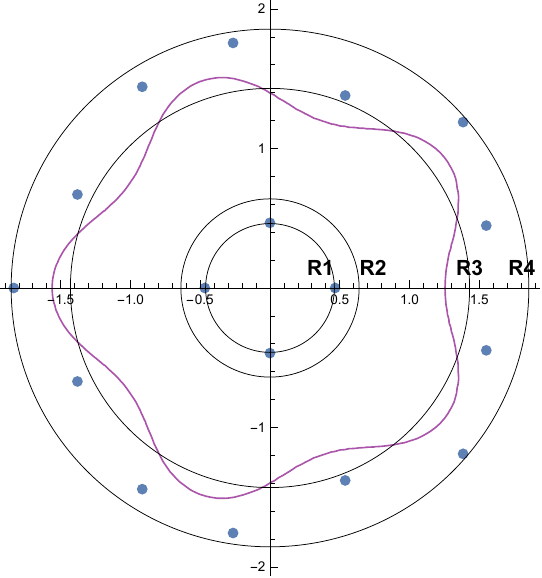}
 \caption{Radii binding location of zeros of $f(z) = z^9 + 7z^4 + 14\overline{z}^4 - 1$}
 \label{annular_bounds}
\end{figure}
\end{remark}

\begin{proof}[Proof of Theorem \ref{Annulus}.]

We now solve for the value of each radius, $R_1$, $R_2$, $R_3$, and $R_4$.

\textbf{$R_1:$} We show no zeros of $f$ exist within radius $R_1 = $ $\frac{1}{(b + a + 1)^\frac{1}{k}}$ of the origin by showing $-1$ is the dominant term of $f$ for any $|z| < R_1$.

Suppose $|z| < \frac{1}{(b + a + 1)^\frac{1}{k}}$, which implies $|z| < 1$. Thus,

\[
|z^n + az^k + b\overline{z}^k | \leq |z|^n + (a + b)|z|^k \leq (1 + a + b)|z|^k.
\]

So, $|-1| > |z^n + az^k + b\overline{z}^k |$ whenever $1 > (b + a + 1)|z|^k$. Therefore, $-1$ is the dominating term of $f$ whenever $|z|< \frac{1}{(1 + a + b)^\frac{1}{k}}$. Thus, by Rouché's Theorem, because $g(z) = -1 \neq 0$ for all $z \in \mathbb{C}$, and $g(z)$ dominates $f$ inside $R_1$, $f$ also has no zeros inside $R_1$.

\textbf{$R_2:$} We work by cases to show that $k$ zeros exist between $R_1$ and $R_2 = \frac{1}{(b-a-1)^\frac{1}{k}}$, and no zeros exist between $R_2$ and $R_3$.

Because we know $b-a > 2$, we have that $b-a-1 > 1$. So, $R_2 < 1$ and $R_3 > 1$. We therefore assume that $|z| < 1$ while solving for $R_2$ and that $|z| > 1$ when solving for $R_3$.

Case 1: Suppose $1 \leq |z| < R_3$. We claim $f$ has no zeros.

To see this, observe that if $1 \leq |z| \leq R_3 = (b-a-1)^\frac{1}{n-k}$, then $1 \leq |z|^{n-k} < (b-a-1)$. Or, equivalently, $|z|^k \leq |z|^n < (b-a-1)|z|^k$.

Hence, for such $z$,

\begin{align*}
|z^n + az^k - 1| &\leq |z|^n + a|z|^k + 1 \\
&\leq |z|^n + a|z|^k + |z|^k \\
&< b|z|^k \\
&= |b\overline{z}^k|.
\end{align*}

Thus, for $1 \leq |z| < R_3$, we find $|f(z)| \geq |b\overline{z}^k| - |z^n + az^k -1| > 0$. Also, for such $z$, $f(z) \neq 0$.

Case 2: Suppose $\frac{1}{(b-a-1)\frac{1}{k}} = R_2 < |z| \leq 1$. Then, we claim $f$ has no zeros.

To see this, note that if $\frac{1}{(b-a-1)^\frac{1}{k}} < |z| \leq 1$, then $1 < |z|^k(b-a-1)$. Thus, for such $z$,

\begin{align*}
|z^n + az^k - 1| &\leq |z|^n + a|z|^k + 1\\
&\leq |z|^k + a|z|^k + 1\\
&< |z|^k + a|z|^k + (b-a-1)|z|^k\\
&= |b\overline{z}^k|.
\end{align*}

Therefore, if $R_2 < |z| \leq 1$, we find $|f(z)| = |z^n + az^k + b\overline{z}^k - 1| \geq |b\overline{z}^k| - |z^n + az^k - 1| > 0$. Hence, for such $z$, $f(z) \neq 0$.

These cases establish that no zeros are in $\{z: R_2 < |z| < R_3\}$. We now would like to use Rouché's Theorem to argue that $f$ has exactly $k$ zeros within the inner annulus.

Recall that we previously showed that if $z$ is on the critical curve, $\frac{k}{n}(b-a) < |z|^{n-k} < \frac{k}{n}(b+a)$. So, if $(b-a) > \frac{n}{k}$, this is outside the unit circle. Because on $|z| = 1$, the above shows $|z^n + az^k - 1| < |b\overline{z}|$, by Rouché's Theorem, the sum of the orders of the zeros in the unit circle is $-k$. Because the entire region is sense-reversing, $f$ has exactly $k$ zeros in the unit circle, which, by the above, must be in $\{z: R_1 < |z| < R_2\}$.

\textbf{$R_3$:} We prove there are no zeros between $R_2$ and $R_3 = (b - a - 1)^\frac{1}{n-k}$ by showing that $b\overline{z}^k$ dominates inside $R_3$.

Recall from Case 2 that because $b-a > 2$, all values of $R_3$ must occur where $|z| > 1$. Recall also that $b\overline{z}^k$ dominates where $b|z|^k > |z|^n + a|z|^k + 1$.

Now, $|z|^n + a|z|^k + 1 \leq |z|^n + (a + 1)|z|^k$. Thus, $b\overline{z}^k$ dominates where $(b - a - 1)|z|^k > |z|^n$. Solving for $|z|$, we find this inequality holds whenever $|z| < (b - a - 1)^\frac{1}{n-k}$. By applying Rouché's Theorem, we find there must still exist $k$ zeros within $R_3$; however, we have already shown that $k$ zeros exist within $R_2$. Hence, no zeros of $f$ exist between $R_2$ and $R_3$.

\textbf{$R_4$}: We prove that $z^n$ dominates where $|z| > (b + a + 1)^\frac{1}{n-k} = R_4$, which shows that all zeros not between $R_1$ and $R_2$ must exist between $R_3$ and $R_4$.

Again, we assume $|z| \geq 1$. Observe that $|az^k + bz^k - 1| \leq (a + b)|z|^k + 1 \leq (a + b + 1)|z|^k$.

Hence, $z^n$ dominates where $|z^n| = |z|^n > (b + a + 1)|z|^k$, which implies $|z| > (b + a + 1)^\frac{1}{n-k}$.

Because $z^n$ dominates for all $|z| > R_4$, we know by Rouché's Theorem that all remaining zeros of $f$ lie between $R_3$ and $R_4$.

Thus, we have shown that $k$ zeros exist between $R_1$ and $R_2$, and the remaining zeros may exist only between $R_3$ and $R_4$. Hence, the two annuli described by the regions between these radii contain all zeros of $f$.
\end{proof}

\section{Conclusion}\label{Conclusion}

We have shown how Rouché's Theorem can be applied to non-circular critical curves to determine the exact number of zeros of a complex harmonic polynomial. It is natural to ask how broadly this approach can be extended to other types of critical curves. Our analysis generalizes similarly whenever \(n\) is a natural number with \(n > k\), including cases where \(k\) is negative (i.e., when $f$ has poles).

In the future, interesting directions include determining exact values for $b_0$ for which different $\lambda$-values determine the exact number of zeros of \(f\), beyond the asymptotic behavior of \(f\) as \(a\) or \(b\) increase toward infinity. This may require additional analytical techniques.

It is also believed that stronger bounds can be formed to locate the zeros, specifically that the bound of $R_2$ in Theorem \ref{Annulus} can be sharpened to $R_2 = \frac{1}{(b+a-1)^{\frac{1}{k}}}$. Moreover, it is suspected that sector bounds could also be found to further reduce the region in the plane where zeros of $f$ might be located.

\backmatter

\bmhead{Acknowledgements}

I would like to thank my research advisor, Dr. Jennifer Brooks, for going above and beyond in her feedback, guidance, and mentorship. I would also like to thank former group members, Alex Lee, Chammelia Zentz, and George Westover, for fruitful discussions during the early stages of forming the results presented in this paper.

\section*{Declarations}

The author certifies that they have no affiliations with or involvement in any organization or entity with any financial interest or non-financial interest in the subject matter or materials discussed in this manuscript.

\newpage

\bibliographystyle{amsplain}

\end{document}